\documentclass{amsart}

\usepackage{amsmath,amsthm,amsfonts,latexsym,amssymb}
\usepackage{setspace}
\usepackage{ifthen}
\usepackage{mathabx}

\ifthenelse{\pdfoutput > 0}%
{\usepackage[pdftex]{graphicx,color}}%
{\usepackage[dvips]{graphicx,color}}

\usepackage[all]{xy}
\CompileMatrices

\usepackage{cancel}

\usepackage[utf8]{inputenc}
\usepackage[T1]{fontenc}

\DeclareMathOperator{\del}{\partial}

\newcommand{\Z}{\mathbb{Z}}

\newcommand{\C}{\mathbb{C}}

\newcommand{\R}{\mathbb{R}}

\newtheorem{theorem}{Theorem}
\newtheorem{corollary}[theorem]{Corollary}
\newtheorem{lemma}[theorem]{Lemma}

\usepackage{hyperref}
\hypersetup{%
  bookmarks=false,%
  pdffitwindow=true,%
  pdftitle={Lagrangian R4s in R8 with Exotic Smooth Structures},%
  pdfauthor={Adam C. Knapp},%
  pdfnewwindow=true,%
  colorlinks=true,%
  citecolor=black,%
  filecolor=black,%
  linkcolor=black,%
  urlcolor=black,%
}

\title[Cotangent Bundles of $4$-Manifolds]%
{Symplectic Structures on the Cotangent Bundles of Open $4$-Manifolds}

\author[A. C. Knapp]{Adam C. Knapp}

\address{Department of Mathematics\\ Columbia University\\ New York, NY 10027}

\email{\href{mailto:knapp@math.columbia.edu}{knapp@math.columbia.edu}}

\thanks{Partially supported by NSF Grant DMS0739392}

\begin{document}
\ifthenelse{\pdfoutput > 0}{\DeclareGraphicsExtensions{.pdf,.jpg,.png,.mps}}%
{\DeclareGraphicsExtensions{.mps,.eps,.ps,.jpg,.png}}

\bibliographystyle{plain}

\begin{abstract}
  We show that, for any two orientable smooth open $4$-manifolds
  $X_0,X_1$ which are homeomorphic, their cotangent bundles
  $T^*X_0,T^*X_1$ are symplectomorphic with their canonical symplectic
  structure. In particular, for any smooth manifold $R$ homeomorphic
  to $\R^4$, the standard Stein structure on $T^*R$ is Stein homotopic
  to the standard Stein structure on $T^*\R^4 = \R^8$. We use this to
  show that any exotic $\R^4$ embeds in the standard symplectic $\R^8$
  as a Lagrangian submanifold.
  As a corollary, we show that $\R^8$ has uncountably many smoothly
  distinct foliations by Lagrangian $\R^4$s with their standard smooth
  structure.
\end{abstract}

\maketitle

\section{Introduction}

{\par We begin with some basics, which can be found in
  \cite{McDuff1998}. Throughout, assume that all manifolds are
  orientable, smoothable, and come with a fixed smooth structure
  unless otherwise indicated. Let $N$ be a smooth manifold of real
  dimension $n$. The cotangent bundle of $N$, $T^*N$, carries a {\em
    canonical $1$-form} $\lambda_{0}$ defined, in local coordinates
  $x_1,\ldots,x_n,y_1=dx_1,\ldots,y_n=dx_n$, by
  $\lambda_0=\sum_{i=1}^n y_i dx_i$. The $1$-form $\lambda_0\in
  \Omega^1(T^*N)$ is uniquely characterized by the property that
  $\sigma^*\lambda=\sigma$ for any $1$-form $\sigma$ on $N$ thought of
  as a section of $T^*N\to N$. Then, for any diffeomorphism
  $\daleth:M\to N$, $\lambda_0^N = \daleth^{**} \lambda_0^M$ where
  $\daleth^{**}$ is the induced map $\daleth^{**}:T^*(T^*M) \to
  T^*(T^*N)$. }

{\par From the canonical $1$-form we obtain a {\em canonical
    symplectic form} $\omega_{0} = -d\lambda_{0}$ on $T^*N$. In local
  coordinates, $\omega_0$ has the form $\omega_{0}=\sum_{i=1}^n
  dx_i\wedge dy_i$. As $\omega_0$ depends only on $\lambda_0$, any two
  diffeomorphic manifolds have symplectomorphic cotangent
  bundles. This fact is the basis of an idea of V. I. Arnol'd which
  states that the smooth topology of a manifold should be reflected in
  the symplectic topology of its cotangent bundle.}

{\par As a realization of this idea, M. Abouzaid showed in
  \cite{MR2874640} that in dimension $n \equiv 1 \mod 4$, an exotic
  $n$-sphere $S$ which does not bound a paralizable manifold does not
  embed as a Lagrangian submanifold of $T^*S^n$ with the the standard
  symplectic structure; hence the cotangent bundle of such an exotic
  sphere cannot be symplectomorphic to $T^*S^n$. As $T^*S$ and
  $T^*S^n$ are diffeomorphic, the canonical symplectic structure on
  $T^*S$ can be considered an exotic symplectic structure on
  $T^*S^n$. }

{\par The existence and classification of exotic symplectic structures
  on a given smooth manifold is of independent interest. In
  \cite{MR2182703} P. Seidel and I. Smith and in \cite{arXiv1007.3281}
  M. Abouzaid and P. Seidel construct exotic symplectic structures on
  $\R^{2n}$ for $n \geq 4$. Also, in \cite{MR2497314}, M. McLean
  constructs exotic symplectic structures on $T^*S^n$ for $n\geq
  3$. However, these constructions do not arise as cotangent bundles,
  nor are they symplectomorphic to them in any obvious way. }

{\par When $n\neq 4$, smoothing theory tells us that there is a unique
  smooth structure on the topological manifold $\R^n$ up to
  homeomorphism. (See \cite{MR0149457} for $n>4$, \cite{MR0048805} for
  $n=3$.) So among the symplectic structures on $T^*\R^n \cong
  \R^{2n}$, there is a fixed one which corresponds to the canonical
  structure on $T^*\R^n$. However in dimension $n=4$, the topological
  manifold $\R^4$ admits uncountably infinitely many inequivalent
  smoothings. For each smooth manifold $R$, homeomorphic to $\R^4$,
  $T^*R$ and $T^*\R^4$ are diffeomorphic. }

{\par It is then natural to ask: Are any of the exotic symplectic
  structures on $\R^8$ symplectomorphic to the canonical structure on
  $T^*R$ for some smooth $R$ homeomorphic to $\R^4$? We answer a
  stronger question in the negative. }

\begin{theorem} \label{thm:moser-def} Let $X_0,X_1$ be smooth, open,
  homeomorphic $4$-manifolds. If $\pi_1(X_i)\neq 0$ we assume that
  there is an $s$-cobordism between $X_0$ and $X_1$. \footnote{Recall
    that an $s$-cobordism is an $h$-cobordism with vanishing Whitehead
    torsion. In dimension $4$, Freedman proved that in the case of
    ``good'' fundamental group, such an $s$-cobordism has a
    topological product structure. See Theorem 7.1A of
    \cite{MR1201584}. We will only require the existence of such an
    $s$-cobordism, not a product structure on it, so we can remove the
    qualification on $\pi_1$. }
  Then each of $T^*X_0$ and $T^*X_1$ have a fixed Stein structure up
  to Stein homotopy and the Stein structures on $T^*X_0$ and $T^*X_1$
  are Stein homotopic. Therefore, $T^*X_0$ and $T^*X_1$ are
  symplectomorphic with their canonical symplectic structures.
\end{theorem}

{\par We reserve the definition of a Stein manifold until the next
  section. A special case of Theorem~\ref{thm:moser-def} is the
  following: }

\begin{corollary} \label{cor:r4}
  Let $R$ be a smooth $4$-manifold homeomorphic to $\R^4$. Then the
  Stein structure on $T^*R$ is Stein homotopic to the standard Stein
  structure on $\R^8$; hence $T^*R$ and $\R^8$ are symplectomorphic.
\end{corollary}

{\par A Lagrangian submanifold $L$ of a symplectic manifold
  $(V,\omega)$ is a submanifold of maximal dimension where
  $\omega|_L\equiv 0$. In a cotangent bundle, graphs of closed
  $1$-forms are Lagrangian. }

\begin{corollary}
  Let $R$ be any exotic $\R^4$. Then $R$ embeds in the standard
  symplectic $(\R^8,\omega)$ as a Lagrangian submanifold.
\end{corollary}

\begin{proof}
  The zero section of $T^*R$ is a Lagrangian copy of $R$ which sits
  inside $\R^8$ as its image under the symplectomorphism.
\end{proof}

{\par Currently, there are no known smooth manifolds which
  \begin{enumerate}
  \item are homeomorphic to $\R^4$,
  \item have finite handlebody decompositions, and
  \item which are known to be not diffeomorphic to the standard
    $\R^4$.
  \end{enumerate}
  If the smooth $4$-dimensional Poincaré conjecture is false, such
  an object exists and arises by puncturing an exotic
  $4$-sphere. Potential examples arise via Gluck twists
  \cite{MR0131877} or from proposed counterexamples to the
  Andrews-Curtis conjecture. Recall that the Andrews-Curtis conjecture
  states that balanced presentations of the trivial group can be
  trivialized using the Andrews-Curtis moves; a collection of moves on
  group presentations related to elementary Morse moves of $1$ and $2$
  handles. (See Remark 5.1.11 of \cite{MR1707327} for examples.) Since
  no such finite handlebody is currently known to be exotic, all known
  examples involve highly complicated behavior at infinity. }

{\par This complicated behavior at infinity obstructs the usual
  definitions of Lagrangian Floer homology for non-compact
  Lagrangians. We then ask the question: Can a Lagrangian Floer
  homology for be defined for Lagrangians such as we have describe? If
  so, can any of the above Lagrangian exotic $\R^4$s be distinguished
  by their Floer homologies? }

{\par A symplectic manifold $(V,\omega)$ is called {\em exact} if $\omega =
  d\alpha$ for some $1$-form $\alpha$. (In the case of the canonical
  structure on a cotangent bundle $\alpha=-\lambda_0$.) A Lagrangian
  submanifold $L$ of an exact symplectic manifold is exact if
  $\alpha|_L$ is exact. Note that $\alpha|_L$ is closed on $L$ as
  $d\alpha|_L=\omega|_L\equiv 0$. When $L$ has $H^1(L;\R)=0$, every
  closed form is exact.}

{\par The usual version of the {\em nearby Lagrangian conjecture}
  states that if a closed manifold $L$ is an exact Lagrangian in
  $T^*N$ (with $N$ compact) then $L$ is Hamiltonian isotopic to the
  zero section of $T^*N$. We have then shown that the corresponding
  non-compact version (when $N$ is $4$-dimensional and open) is false
  -- without some sort of control at infinity -- since such a
  Hamiltonian isotopy would give a diffeomorphism between any two
  smooth structures on $N$. }

{\par Let $(X_0,\mathcal{F}_0)$ and $(X_1,\mathcal{F}_1)$ be two
  smooth manifolds with smooth foliations. We will call these two
  foliations {\em smoothly equivalent} if there exists a
  diffeomorphism $\phi:X_0\to X_1$ such that
  $\phi_*(\mathcal{F}_0)=\mathcal{F}_1$. Suppose that a foliated
  manifold $(X,\mathcal{F})$ admits a {\em smooth global slice} by
  which we will mean a smooth manifold $Y$ and smooth embedding $Y \to
  X$ which intersects every leaf of $\mathcal{F}$ once
  transversely. It is straightforward to show that any two global
  slices of $\mathcal{F}$ are diffeomorphic; hence the leaf space of
  $\mathcal{F}$ is naturally identified with $Y$ with its smooth
  structure. }

\begin{corollary}
  The standard symplectic $\R^8$ admits uncountably infinitely many
  smoothly distinct foliations by Lagrangian $\R^4$'s with the standard
  smooth structure.
\end{corollary}

\begin{proof}
  For each exotic $\R^4$, $R$, $T^*R$ has a codimension $4$ foliation
  by the fibers of the projection $T^*R \to R$, each a Lagrangian
  $\R^4$ with its standard smooth structure. Since the leaf space
  $\mathcal{L}$ is naturally identified with $R$, $\mathcal{L}$ is a
  smooth manifold and its smooth type is an invariant of the
  foliation. The result follows by the uncountability of smooth
  structures on $\R^4$ together with Corollary~\ref{cor:r4}.
\end{proof}

\section{Stein Manifolds}

{\par A smooth manifold $V$ of real dimension $2n$, equipped with an
  almost complex structure $J$ is said to be Stein if $J$ is
  integrable and $V$ admits a proper holomorphic embedding into
  $\C^{N}$ for some $N$. By \cite{MR0098847,MR0123732,MR0148942}, a
  complex manifold $(V,J)$ is Stein if and only if it admits a smooth
  function $\phi:V \to \R$ which is
  \begin{enumerate}
  \item proper and bounded below (exhausting) and
  \item is $J$-convex in the sense that $-dd^\C\phi(v,Jv)>0$ for all
    $v$. Here $d^\C \phi$ denotes $d\phi \circ J$.
  \end{enumerate}
  We call the triple $(V,J,\phi)$ a Stein structure on $V$. Note that
  $-dd^\C\phi$ is a symplectic form on $V$ compatible with $J$. In
  fact, the existence of a Stein structure only requires a weaker
  condition, due to the following theorem of Eliashberg: }

\begin{theorem}[Eliashberg \cite{MR1044658}] \label{thm:E}
  A smooth manifold $V^{2n}$ with $2n>4$ admits a Stein structure if
  and only if it admits an almost complex structure $J$ and an
  exhausting Morse function $\phi$ with critical points of index
  $\leq n$. More precisely, $J$ is homotopic through almost complex
  structures to a complex structure $J'$ such that $\phi$ is $J'$
  convex.
\end{theorem}

{\par Associated to every symplectic manifold $(V,\omega)$ there is a
  contractible space of almost complex structures $J$ which are
  compatible with $\omega$ in the sense that $g(v,w)=\omega(v,Jw)$ is
  a Riemannian metric and $\omega(Jv,Jw)=\omega(v,w)$. When
  $(V,\omega)=(T^*N,\omega_0)$, we can construct a contractible
  subspace of these structures explicitly in terms of Riemannian
  metrics on $N$. That is, pick a Riemannian metric $g_N$ on $N$ and
  define $J_g$ in local coordinates $x_1,\ldots,x_n,y_1,\ldots,y_n$
  (for $T^*N$) by
  \begin{displaymath}
    J_g\left(\frac{\del}{\del x_i} \right) 
    =
    \sum_{j=1}^n (g_N)_{ij} \frac{\del}{\del y_j}
    \text{ and }
    J_g\left(\frac{\del}{\del y_i} \right)
    =
    \sum_{j=1}^n -(g_N)^{ij} \frac{\del}{\del x_j}
  \end{displaymath}
  Then $J_g$ is compatible with $\omega_0$ and, on the zero section
  $Z$, $g|_Z \equiv g_N$. This almost complex structure is not, in
  general, integrable.
}

{\par Now, pick an exhausting Morse function $f:N \to \R$. Let
  $\phi:T^*N \to \R$ be defined by $\phi(x,y)= f(x) +
  \frac{1}{2}\|y\|^2$. Then $\phi$ is again an exhausting Morse
  function, now for $T^*N$, whose critical points occur along the zero
  section and have an index-preserving bijection with those of $f$. As
  long as the dimension of $N$ is $n>2$, the conditions of
  Theorem~\ref{thm:E} are satisfied and $T^*N$ admits a Stein
  structure. }

{\par An exhausting Morse function on $V^{2n}$ is called {\em
    subcritical} if it has only critical points of index $< n$. A
  Stein structure $(\phi,J)$ is called {\em subcritical} if $\phi$ is
  subcritical. In the subcritical case, Y. Eliashberg and K. Cieliebak
  showed that Stein structures are unique up to homotopy:
  \begin{theorem}[Eliashberg, Cieliebak \cite{EliashbergCieliebak}]
    \label{thm:EC}
    Let $n>3$ and let $(\phi_0,J_0)$, $(\phi_1,J_1)$ be two
    subcritical Stein structures on $V^{2n}$. If $J_0$ and $J_1$ are
    homotopic as almost complex structures, then $(\phi_0,J_0)$ and
    $(\phi_1,J_1)$ are homotopic as Stein structures.
  \end{theorem}
  Here a Stein homotopy consists of a concatenation of ``simple Morse
  homotopies'' i.e. sequences of Morse birth-deaths and handle
  slides. In this case, critical points of the $\phi_t$ do not escape
  to infinity and Moser's trick\footnote{See Section 3.2
    of~\cite{McDuff1998} for the compact case.} applies to give us a
  $1$-parameter family of diffeomorphisms taking one Stein structure
  to the other. See \cite{\cite{EliashbergCieliebak}} for details.
  Consequently, the underlying symplectic manifolds for the Stein
  structures $(\phi_0,J_0)$ and $(\phi_1,J_1)$ are symplectomorphic. }

{\par Before we begin with the proof of
  Theorem~\ref{thm:moser-def}. We show the following:}
\begin{lemma} \label{lem:cotangent} If two $4$-manifolds $X_0,X_1$ are
  homeomorphic, then their cotangent bundles are diffeomorphic as
  $8$-manifolds. If $\pi_1(X_i)\neq 0$, we assume that the $X_i$ are
  $s$-cobordant.
\end{lemma}

\begin{proof}
  Let $W$ be an $h$-cobordism between $X_0$ and $X_1$. If
  $\pi_1(X_i)\neq 0$, then assume that the Whitehead torsion of $W$
  vanishes so that $W$ is an $s$-cobordism. There is a rank $4$ real
  vector bundle $T$ on $W$ which restricts to $T^*X_0$ and $T^*X_1$ on
  $\partial W$. The bundle $T$ can be obtained by pulling back
  $T^*X_0$ via the homotopy equivalence $W\to X_0$. Then we see that
  $T|_{X_1}$ is isomorphic to $T^*X_1$ by noting that it has the
  requisite characteristic classes.

  The unit sphere bundle $S(T)$ is then an $s$-cobordism between
  $7$-manifolds -- the unit sphere bundles of $T^*X_i$ -- and, by the
  $s$-cobordism theorem, a product. 

  Now, taking the unit disc bundle $D(T)$, we get an $s$-cobordism of
  $8$-manifolds with boundary.  As we have seen, it is a product on
  the boundary. As this is again a product by the $s$-cobordism
  theorem, the diffeomorphism of the $T^*X_i$ follows by restricting
  to the interior.
\end{proof}

\begin{proof}[Proof of Theorem~\ref{thm:moser-def}]

  If $X_0$ and $X_1$ are homeomorphic, then $T^*X_0$ and $T^*X_1$ are
  diffeomorphic by Lemma~\ref{lem:cotangent}. Choose some
  representative $V$ of this diffeomorphism type and particular
  diffeomorphisms $\daleth_i:V \to T^*X_i$. The canonical $1$-forms,
  symplectic forms and choices of almost complex structures for the
  cotangent bundles pull back to $\lambda_i$, $\omega_i$ and $J_i$ on
  $V$.

  First, we show that $J_0$ and $J_1$ are homotopic as almost complex
  structures. Write $\mathcal{I}(n)$ for the space
  $GL^+(2n,\R)/GL(n,\C)$ which classifies almost complex structures on
  $\R^{2n}$. The space $\mathcal{I}(n)$ is homotopy equivalent to
  $SO(2n)/U(n)$ and, when $n=4$, we can compute several of the
  homotopy groups:
  \begin{eqnarray*}
    \pi_0 \mathcal{I}(4) = 0,  &
    \pi_1 \mathcal{I}(4) = 0,  & 
    \pi_2 \mathcal{I}(4) = \Z, \\
    \pi_3 \mathcal{I}(4) = 0, & 
    \pi_4 \mathcal{I}(4) = 0, & 
    \pi_5 \mathcal{I}(4) = 0, \\
    \pi_6 \mathcal{I}(4) = \Z, & & 
  \end{eqnarray*}
  The obstructions to a homotopy between $J_0$ and $J_1$ lie in
  $H^i(V,V^{(i-1)};\pi_i\mathcal{I}(4))$ where $V^{(i-1)}$ is the
  $i-1$--skeleton of $V$. As $H^i(V;\Z)=0$ for $i\geq 4$, the only
  non-trivial group in this list is
  $H^2(V,V^{(1)};\pi_2\mathcal{I}(4))\cong H^2(V,V^{(1)};\Z) \cong
  H^2(X_i,X_i^{(1)})$. However, $J_0$ and $J_1$ are homotopic over the
  $2$-skeleton so this obstruction vanishes. (Clearly they both have
  the same first Chern class.) Therefore, $J_0$ and $J_1$ are
  homotopic as almost complex structures.

  As we saw above, for each Morse function $f_i$ on $X_i$, we obtain
  Morse functions $\phi_i$ on $T^*X_i$ whose critical points (and,
  once a metric is chosen, flow lines along the zero section) can be
  identified with those of $f_i$. Note that the $\phi_i$ will be
  proper and bounded below when the $f_i$ are.

  In order to apply Theorem~\ref{thm:EC}, we will need to show that
  the $X_i$ admit Morse functions without index $4$-critical
  points. To construct such a Morse function, begin with a Morse
  function which is bounded below and whose critical values are
  discrete, with a single critical point per critical value.  With
  such a choice, $X_i$ is identified with a composition of elementary
  cobordisms.

  Let $x_4$ be an index $4$ critical point. Then the unstable manifold
  of $x_4$ must include trajectories to at most finitely many index
  $3$ critical points. Further, if $x_3$ is an index $3$ critical
  point, its stable manifold must meet at most $2$ index $4$ critical
  points. (Since the stable manifold is $1$ dimensional.) As
  $H_4(X_i;\Z)=0$, the boundary map $\del:C_4(X)\to C_3(X)$ is
  injective; so, after possibly performing some handle-slides, we can
  cancel any index $4$ critical point with an index $3$ critical
  point. Therefore, each of the $X_i$ admit a Morse function $f_i$ without
  index $4$-critical points.

  Let $(\phi_i,J_i)$ be the Stein structures on $V$ constructed using
  Theorem~\ref{thm:E} where $J_i$ is a complex structure homotopic to
  the almost complex structure $J'_i$ constructed from the metric on
  $X_i$. As each of the $\phi_i$ are subcritical, we then apply
  Theorem~\ref{thm:EC} to see that the $(\phi_i,J_i)$ are Stein
  homotopic. \end{proof}


\end{document}